\documentclass{article}
\usepackage{amssymb,amsmath,bm,geometry,graphics,graphicx,url,color}

\def\no{\noindent}
\def\pmatrix{\left(\begin{array}}
\def\endpmatrix{\end{array}\right)}

\newtheorem{theo}{Theorem}
\newtheorem{lem}{Lemma}

\newtheorem{rem}{Remark}
\newtheorem{defi}{Definition}

\title{Arbitrary-order  functionally fitted energy-diminishing methods for gradient systems}

\author{Bin Wang\,
\footnote{School of Mathematical Sciences, Qufu Normal University,
Qufu 273165, P.R. China; Mathematisches Institut, University of
T\"{u}bingen, Auf der Morgenstelle 10, 72076 T\"{u}bingen, Germany.
The research is supported in part by the Alexander von Humboldt
Foundation and by the Natural Science Foundation of Shandong
Province (Outstanding Youth Foundation) under Grant ZR2017JL003.
E-mail:~{\tt wang@na.uni-tuebingen.de} } \and Ting Li\thanks{School
of Mathematical Sciences, Qufu Normal University, Qufu 273165, P.R.
China. E-mail:~{\tt 1009587520@qq.com}} \and Yajun Wu\thanks{School
of Mathematical Sciences, Qufu Normal University, Qufu 273165, P.R.
China. E-mail:~{\tt 1921170786@qq.com}} }

\begin{document}
\maketitle

\begin{abstract} It is well known that for gradient systems in Euclidean space or on
a Riemannian manifold, the energy decreases monotonically along
solutions. In this letter we derive and analyse functionally fitted
energy-diminishing methods to preserve this key property of gradient
systems. It is proved that the novel methods are unconditionally
energy-diminishing and can achieve damping for very stiff gradient
systems. We also show that the methods can be of arbitrarily high
order and discuss their implementations. A numerical test is
reported to illustrate the efficiency of the  new methods in
comparison with three existing numerical methods in the literature.

\medskip
\no{\bf Keywords:}  gradient systems, energy-diminishing methods,
functionally fitted methods, arbitrary-order methods

\medskip
\no{\bf MSC:}65L05

\end{abstract}

\section{Introduction}
In this letter, we investigate   the  following  gradient systems in
coordinates:
\begin{equation}\label{IVPPP-G}
G(y(t))\dot{y}(t)=- \nabla U(y(t)),\quad
y(0)=y_{0}\in\mathbb{R}^{d},\quad t\in[0,T],
\end{equation}
where $U(y): \mathbb{R}^d\rightarrow \mathbb{R}$ is a potential
function  and the symmetric matrix $G(y)$ is assumed to satisfy
\begin{equation*}\label{gy}
 v^{\intercal} G(y) v \geq  v^{\intercal} \Gamma v>0
\end{equation*}
for all vectors $v\neq 0.$ Here $\Gamma $ is a fixed positive
definite matrix.

Gradient systems frequently  arise in a wide variety of applications
both  in a finite-dimensional and infinite-dimensional setting.
There are many examples of this system (see, e.g.
\cite{Bao04,Chen02,Droske04,Michailovich07,Otto01,Penzler12,Strzodka04})
such as models in quantum systems,  in differential geometry,  in
image processing, and in material science. A fundamental and key
property of gradient systems is that along every exact solution of
\eqref{IVPPP-G},  one has
\begin{equation*}\label{energy-dis}
 \frac{d}{dt}U(y(t))= \nabla
 U(y(t))^{\intercal}\dot{y}(t)=-\dot{y}(t)^{\intercal}G(y(t))\dot{y}(t)\leq0,
\end{equation*}
which shows that $U(y(t))$ is monotonically decreasing  with $t$.
The monotonicity is true with strict inequality except at stationary
points of $U$. The aim of this letter is to formulate and analyse a
novel kind of methods preserving this monotonicity in the numerical
treatment, i.e., after one step of the method starting from $y_0$
with a time step $h$ one would have $$U(y_1)\leq U(y_0).$$

In order to get methods with this property,  Hairer and Lubich
 analysed various energy-diminishing  methods in \cite{hairer2013}.
 They showed that implicit Euler method has this property but it  is only of order
 one.  Algebraically stable Runge-Kutta methods were proved to reduce the energy in each
step under a mild step-size restriction, which means that
Runge-Kutta methods are not unconditionally energy-diminishing.
 They also showed that discrete-gradient methods,  averaged vector field (AVF) methods and AVF collocation
methods are unconditionally energy-diminishing, but cannot achieve
damping for very stiff gradient systems. In this letter, we will
derive a novel kind of methods which  can  be
 of arbitrarily high order.
 Moreover, the methods will be shown that they
 are unconditionally energy-diminishing  and are strongly damped even
for very stiff gradient systems.

The rest of this letter is organised as follows. In Section
\ref{sec-fomu}, we derive the novel methods and prove that they are
unconditionally
 energy-diminishing for gradient systems. The unconditionally damping
property is analysed in  Section \ref{sec-Damping}.    We study the
order of the methods in Section \ref{sec-order}. Section
\ref{sec-implementations} is devoted to the implementation issue. In
Section \ref{sec-nume}, a numerical test is carried out to
 demonstrate the   excellent qualitative behavior.
Section \ref{sec-conc} focuses on the concluding remarks.

\section{Functionally fitted energy-diminishing methods}\label{sec-fomu}
In order to formulate the novel methods, we will use the
functionally fitted technology, which  is a popular approach to
constructing  efficient and effective methods in scientific
computing (see, e.g. \cite{Li_Wu(na2016),17-new}). To this end,
  define a function space
$Y$=span$\left\{\varphi_{0}(t),\ldots,\varphi_{r-1}(t)\right\}$ on
$[0,T]$
  by (see \cite{Li_Wu(na2016)})
\begin{equation*}
Y=\left\{w : w(t)=\sum_{i=0}^{r-1}\varphi_{i}(t)W_{i},\ t\in I,\
W_{i}\in\mathbb{R}^{d}\right\},
\end{equation*}
where $\{\varphi_{i}(t)\}_{i=0}^{r-1}$ are sufficiently smooth and
linearly independent on $[0,T]$.  In this letter, we consider   two
finite-dimensional function spaces  $Y$ and $X$ as follows
\begin{equation*}\label{FF}
Y=\text{span}\left\{\varphi_{0}(t),\ldots,\varphi_{r-1}(t)\right\},\quad
X=\text{span}\left\{1,\int_{0}^{t}\varphi_{0}(s)ds,\ldots,\int_{0}^{t}\varphi_{r-1}(s)ds\right\}.
\end{equation*}
Choose a stepsize $h$ and define the function spaces $Y_{h}$ and
$X_{h}$ on $[0,1]$ by
\begin{equation*}
Y_{h}=\text{span}\left\{\tilde{\varphi}_{0}(\tau),\ldots,\tilde{\varphi}_{r-1}(\tau)\right\},\quad
X_{h}=\text{span}\left\{1,\int_{0}^{\tau}\tilde{\varphi}_{0}(s)ds,\ldots,\int_{0}^{\tau}\tilde{\varphi}_{r-1}(s)ds\right\},
\end{equation*}
 where $\tilde{\varphi}_{i}(\tau)=\varphi_{i}(\tau h),\ \tau\in[0,1]$
for $i=0,1,\ldots,r-1$. We remark that  for all the functions
throughout this letter,  the notation $\tilde{f} (\tau)$ is referred
to $f(\tau h)$.

We will use a projection $\mathcal{P}_{h}$  in the formulation of
the new methods.  It was defined in \cite{Li_Wu(na2016)} and we
summarise it here.
\begin{defi}\label{def projection} (See \cite{Li_Wu(na2016)})
Let $\mathcal{P}_{h}\tilde{w}$ be a projection of $\tilde{w}$ onto
$Y_{h}$, where $\tilde{w}(\tau)$ is a continuous
$\mathbb{R}^{d}$-valued function on $[0,1]$.
  The definition of
$\mathcal{P}_{h}\tilde{w}$ is given by
\begin{equation*}
\langle
\tilde{v}(\tau),\mathcal{P}_{h}\tilde{w}(\tau)\rangle=\langle
\tilde{v}(\tau),\tilde{w}(\tau)\rangle\quad \text{for any}\ \
\tilde{v}(\tau)\in Y_{h},
\end{equation*}
where   the inner product $\langle\cdot,\cdot\rangle$ is defined by
\begin{equation*}
\langle \tilde{w}_{1},\tilde{w}_{2}\rangle=\langle
\tilde{w}_{1}(\tau),\tilde{w}_{2}(\tau)\rangle_{\tau}=\int_{0}^{1}\tilde{w}_{1}(\tau)\cdot
\tilde{w}_{2}(\tau)d\tau.
\end{equation*}
Here $\tilde{w}_{1}$ and $\tilde{w}_{2}$ are two integrable
functions (scalar-valued or vector-valued)
 on $[0,1]$, and  `$\cdot$' denotes the entrywise multiplication
operation if they are  both vector-valued functions.
\end{defi}

The  following property of  $\mathcal{P}_{h}$ will also be useful in
this letter, which has been proved in \cite{Li_Wu(na2016)}.

\begin{lem}\label{proj lem}(See \cite{Li_Wu(na2016)}) The projection $\mathcal{P}_{h}\tilde{w}$ can be explicitly expressed as
\begin{equation*}
\mathcal{P}_{h}\tilde{w}(\tau)=\langle
P_{\tau,\sigma},\tilde{w}(\sigma)\rangle_{\sigma},
\end{equation*}
where
\begin{equation*}\label{PCOEFF}
P_{\tau,\sigma}=\sum_{i=0}^{r-1}\tilde{\psi}_{i}(\tau)\tilde{\psi}_{i}(\sigma),
\end{equation*}
and $\left\{\tilde{\psi}_{0},\ldots,\tilde{\psi}_{r-1}\right\}$ is a
standard orthonormal basis of $Y_{h}$ under the inner product
$\langle\cdot,\cdot\rangle$.
\end{lem}

Based on these preliminaries, we are in a position to present the
scheme of functionally fitted energy-diminishing methods.

\begin{defi}\label{def EDCr}
Choose a stepsize $h$  and consider a function $\tilde{u}\in X_{h}$
with $\tilde{u}(0)=y_{0}$, satisfying
\begin{equation} \label{EDCr}
\begin{aligned} &
G(\tilde{u}(\tau))\tilde{u}^{\prime}(\tau) = -\mathcal{P}_{h}\big(
\nabla U(\tilde{u}(\tau))\big).
\end{aligned}
\end{equation}
The numerical approximation  after one step is  defined by $
y_{1}=\tilde{u}(1).$ We call this method     as functionally-fitted
energy-diminishing  method and denote it by FFED.
\end{defi}

\begin{theo}\label{ED}
The  FFED method \eqref{EDCr} reduces the energy of gradient systems
\eqref{IVPPP-G}, i.e., $U(y_{1})\leq U(y_{0}).$ This means that our
FFED method \eqref{EDCr} is unconditionally energy-diminishing.
\end{theo}
\begin{proof}According to the definitions of $X_h$ and $Y_h$, we know that if
$\tilde{u}(\tau)\in X_h$, then one has $\tilde{u}'(\tau)\in Y_h$.
From the definition of $\mathcal{P}_{h}$, it follows that
\begin{equation*}
\int_{0}^{1}\tilde{u}'(\tau)_{i}\mathcal{P}_{h}\big( \nabla
U(\tilde{u}(\tau))\big)_{i}d\tau=\int_{0}^{1}\tilde{u}'(\tau)_{i}
\big( \nabla U(\tilde{u}(\tau))\big)_{i}d\tau,\quad i=1,2,\ldots,d,
\end{equation*}
where   $(\cdot)_{i}$ denotes the $i$th entry of a vector. Then, we
arrive at
\begin{equation*}
\int_{0}^{1}\tilde{u}'(\tau)^{\intercal}\mathcal{P}_{h}\big( \nabla
U(\tilde{u}(\tau))\big)
d\tau=\int_{0}^{1}\tilde{u}'(\tau)^{\intercal}   \nabla
U(\tilde{u}(\tau)) d\tau.
\end{equation*}
Therefore, it is obtained that
\begin{equation*}
\begin{aligned}
&U(y_{1})-U(y_{0})=\int_{0}^{1}\frac{d}{d\tau}U(\tilde{u}(\tau))d\tau
 =h\int_{0}^{1}\tilde{u}'(\tau)^{\intercal}\nabla
U(\tilde{u}(\tau))d\tau=h\int_{0}^{1}\tilde{u}'(\tau)^{\intercal}\mathcal{P}_{h}\big(
\nabla U(\tilde{u}(\tau))\big)d\tau.
\end{aligned}
\end{equation*}
Inserting the scheme \eqref{EDCr} into this formula yields
\begin{equation*}
\begin{aligned}
U(y_{1})-U(y_{0})&=-h\int_{0}^{1}\tilde{u}'(\tau)^{\intercal}G(\tilde{u}(\tau))\tilde{u}^{\prime}(\tau)d\tau\leq0.
\end{aligned}
\end{equation*}
 \end{proof}

\section{Unconditionally   damping
property}\label{sec-Damping} In this section, we consider the
 potential $U$ of  the form
\begin{equation}\label{spe-U}
  U(y)=\dfrac{1}{2}y^{\intercal}Ay+V(y),
\end{equation}
  where the function $V$  is  twice continuously differentiable  and the matrix $A$  is   symmetric positive semi-definite
and of arbitrarily large norm. In this case, \eqref{IVPPP-G} is a
stiff gradient system.   This kind of stiff systems arise from the
spatial discretization of Cahn--Hilliard and Allen--Cahn partial
differential equations (see, e.g. \cite{Barrett02,Feng04}). Many
effective methods have been derived for this stiff gradient system
with a constant matrix $G$ and we refer to
\cite{hairer2006,Hairer93-stiff,Li_Wu(sci2016),Liu_Wu(AML),wang-2016,17-new,wu2017-JCAM,wang2017-JCM,wu2013-book}
for example.

 The FFED method \eqref{EDCr} for solving this stiff gradient system is defined  as
 follows.
\begin{defi}\label{def EDCr-E}
We consider a function $\tilde{u}\in X_{h}$ with
$\tilde{u}(0)=y_{0}$, satisfying
\begin{equation} \label{EDCr-E}
\begin{aligned} &G(\tilde{u}(\tau))
\tilde{u}^{\prime}(\tau)+A\tilde{u}(\tau)=-\mathcal{P}_{h}\big(
\nabla V(\tilde{u}(\tau))\big),\qquad y_{1}=\tilde{u}(1).
\end{aligned}
\end{equation}
This method is denoted  by EFFED.
\end{defi}

\begin{theo}\label{ED-E}
The EFFED  method \eqref{EDCr-E} reduces the energy of the stiff
gradient system \eqref{IVPPP-G} with \eqref{spe-U}, i.e.,
$U(y_{1})\leq U(y_{0}).$
\end{theo}
\begin{proof}
This proof is similar to that of Theorem \ref{ED}. For any
$\tilde{w} (\tau)\in Y_{h}$, it is easy to prove that
\begin{equation}\label{inner}
\begin{aligned}
&\int_{0}^{1}\tilde{w}(\tau)^{\intercal}G(\tilde{u}(\tau))\tilde{u}^{\prime}(\tau)d\tau=-\int_{0}^{1}\tilde{w}(\tau)^{\intercal}A\tilde{u}(\tau)d\tau
-\int_{0}^{1}\tilde{w}(\tau)^{\intercal} \mathcal{P}_{h}\big( \nabla
V(\tilde{u}(\tau))\big) d\tau\\
&=-\int_{0}^{1}\tilde{w}(\tau)^{\intercal}A\tilde{u}(\tau)d\tau
-\int_{0}^{1}\tilde{w}(\tau)^{\intercal}  \nabla
V(\tilde{u}(\tau)) d\tau= -\int_{0}^{1}\tilde{w}(\tau)^{\intercal} \nabla U(\tilde{u}(\tau)) d\tau.\\
\end{aligned}
\end{equation}
Considering $\tilde{u}^{\prime}(\tau)\in Y_{h}$ and letting
$\tilde{w}(\tau)= \tilde{u}^{\prime}(\tau)$ in \eqref{inner}, we
obtain
\begin{equation*}\label{inner-1}
\begin{aligned}
&\int_{0}^{1}\tilde{u}^{\prime}(\tau)^{\intercal}G(\tilde{u}(\tau))\tilde{u}^{\prime}(\tau)d\tau=- \int_{0}^{1}\tilde{u}^{\prime}(\tau)^{\intercal} \nabla U(\tilde{u}(\tau)) d\tau.\\
\end{aligned}
\end{equation*}
Thus, one arrives
\begin{equation*}
\begin{aligned}
U(y_{1})-U(y_{0})&=\int_{0}^{1}\frac{d}{d\tau}U(\tilde{u}(\tau))d\tau
=h\int_{0}^{1}\tilde{u}'(\tau)^{\intercal}\nabla
U(\tilde{u}(\tau))d\tau
=-h\int_{0}^{1}\tilde{u}^{\prime}(\tau)^{\intercal}G(\tilde{u}(\tau))\tilde{u}^{\prime}(\tau)d\tau\leq0.
\end{aligned}
\end{equation*}
 \end{proof}

For $G(y)=I$ and a quadratic potential (i.e., $V(y)=0$ in
\eqref{spe-U}),  our   EFFED method \eqref{EDCr-E} becomes $$
\tilde{u}^{\prime}(\tau)+A\tilde{u}(\tau)=0,$$ which leads to $$
y_{1}=\tilde{u}(1)=e^{ -h  A}y_{0}.$$ This scheme has been presented
and researched in \cite{17-new}. Rewrite this scheme as $
y_{1}=R(-hA)y_0$ with  the stability function $R(-hA)=e^{ -h  A}.$
It is noted that the damping property
 $ |R(\infty)| <1$ plays an important role in   the  properties of
Runge-Kutta methods when solving  semilinear parabolic equations,
which has been researched in \cite{Lubich93,Lubich96}. The role of
the condition  $ |R(\infty)| =0$ has been well understood in Chapter
VI of \cite{Hairer93-stiff}. It has been shown in \cite{hairer2013}
that all the unconditionally energy-diminishing methods including
discrete-gradient methods, AVF methods and AVF collocation methods
show no damping for very stiff gradient systems since they do not
satisfy damping property unconditionally.
  However,  it is noted  that  for our EFFED method \eqref{EDCr-E}, one  has $$ |R(\infty)| =|e^{
-\infty}|=0.$$ This means that our  methods have unconditionally
damping property, which  is a significant feature especially  for
very stiff gradient systems.

\section{Algebraic order}\label{sec-order}
In this section, we analyse the algebraic order of the new methods.

\begin{theo}\label{order}
The FFED method \eqref{EDCr} and EFFED method \eqref{EDCr-E} are
both of order $2r$, which means that
\begin{equation*} \begin{aligned}
&\tilde{u}(1)-y(t_{0}+h)=\mathcal{O}(h^{2r+1}).\end{aligned}
\end{equation*}
\end{theo}
\begin{proof} Denote by $y(\cdot,\tilde{t},
\tilde{y})$ the solution of $y'=-G^{-1}(y) \nabla U(y)$ satisfying
the initial condition $y(\tilde{t},\tilde{t}, \tilde{y})=\tilde{y}$
for any given $\tilde{t}\in[0,h]$. Let $ \Phi(s,\tilde{t},
\tilde{y})=\frac{\partial y(s,\tilde{t}, \tilde{y})}{\partial
\tilde{y}}. $ Recalling the elementary theory of ordinary
differential equations, the following standard result is obtained
\begin{equation*}
\frac{\partial y(s,\tilde{t}, \tilde{y})}{\partial
\tilde{t}}=-\Phi(s,\tilde{t}, \tilde{y})\big(-G^{-1}(\tilde{y})
\nabla U(\tilde{y})\big).
\end{equation*}
On the base of this result, we get
\begin{equation*}
\begin{aligned}
&\tilde{u}(1)-y(t_{0}+  h)=y(t_{0}+  h,t_{0}+  h,\tilde{u}(1))-y(t_{0}+  h,t_{0},y_{0})\\
&=\int_{0}^{1}\frac{d}{d\alpha}y(t_{0}+  h,t_{0}+\alpha h,\tilde{u}(\alpha))d\alpha\\
&=\int_{0}^{1}(h\frac{\partial y}{\partial\tilde{t}}(t_{0}+
h,t_{0}+\alpha h,\tilde{u}(\alpha))+
\frac{\partial y}{\partial\tilde{y}}(t_{0}+  h,t_{0}+\alpha h,\tilde{u}(\alpha))h\tilde{u}^{\prime}(\alpha))d\alpha\\
&=\int_{0}^{1}\Big(-h\frac{\partial y}{\partial\tilde{y}}(t_{0}+
h,t_{0}+\alpha
h,\tilde{u}(\alpha))G^{-1}(\tilde{u}(\alpha))\big(A\tilde{u}(\alpha)- \nabla U(\tilde{u}(\alpha))\big)\\
&\quad+ \frac{\partial y}{\partial\tilde{y}}(t_{0}+ h,t_{0}+\alpha
h,\tilde{u}(\alpha))G^{-1}(\tilde{u}(\alpha))\big(hA
\tilde{u}(\alpha)-h\langle
P_{1,\sigma}, \nabla U(\tilde{u}(\alpha))\rangle_{\alpha}\big)\Big)d\alpha\\
&=-h\int_{0}^{1}\Phi^{1}(\alpha)G^{-1}(\tilde{u}(\alpha))\big(g(\tilde{u}(\alpha))-\mathcal{P}_{h}(g\circ
\tilde{u})(\alpha)\big)d\alpha,
\end{aligned}
\end{equation*}
where $\Phi^{1}(\alpha)=\frac{\partial y}{\partial\tilde{y}}(t_{0}+
h,t_{0}+\alpha h,\tilde{u}(\alpha)).$
 We denote the matrix $\Phi^{1}(\alpha)G^{-1}(\tilde{u}(\alpha))$  by $K(\alpha)$ and
 partition it as
$K(\alpha)=(K_{1}(\alpha),\ldots,K_{d}(\alpha))^{\intercal}$. It
follows from Lemma 3.4 of  \cite{Li_Wu(na2016)} that
\begin{equation*}
K_{i}(\alpha)=\mathcal{P}_{h}K_{i}(\alpha)+\mathcal{O}(h^{r}),\quad
i=1,2,\ldots,d,
\end{equation*}
and
$$g(\tilde{u}(\alpha))-\mathcal{P}_{h}(g\circ
\tilde{u})(\alpha)=\mathcal{O}(h^{r}).$$
 On the other hand, in light of the definition of $\mathcal{P}_{h}$, we
 obtain
\begin{equation*}
\begin{aligned}
&\int_{0}^{1}(\mathcal{P}_{h}K_{i}(\alpha))^{\intercal}g(\tilde{u}(\alpha))d\alpha
=\int_{0}^{1}(\mathcal{P}_{h}K_{i}(\alpha))^{\intercal}\mathcal{P}_{h}(g\circ
\tilde{u})(\alpha)d\alpha,\quad i=1,2,\ldots,d.
\end{aligned}\end{equation*}
 Therefore, one has
\begin{equation*}
\begin{aligned}
&\tilde{u}(1)-y(t_{0}+h)
= -h\int_{0}^{1}\left(\left(\begin{array}{c}(\mathcal{P}_{h}K_{1}(\alpha))^{\intercal}\\
\vdots\\(\mathcal{P}_{h}K_{d}(\alpha))^{\intercal}\end{array}\right)+\mathcal{O}(h^{r})\right)\big(g(\tilde{u}(\alpha))-\mathcal{P}_{h}(g\circ \tilde{u})(\alpha)\big)d\alpha\\
=&-h\int_{0}^{1}\left(\begin{array}{c}(\mathcal{P}_{h}K_{1}(\alpha))^{\intercal}\big(g(\tilde{u}(\alpha))-\mathcal{P}_{h}(g\circ \tilde{u})(\alpha)\big)\\
\vdots\\(\mathcal{P}_{h}K_{d}(\alpha))^{\intercal}\big(g(\tilde{u}(\alpha))-\mathcal{P}_{h}(g\circ
\tilde{u})(\alpha)\big)\end{array}\right)d\alpha
-h\int_{0}^{1}\mathcal{O}(h^{r})\times\mathcal{O}(h^{r})d\alpha\\
=&0+\mathcal{O}(h^{2r+1})=\mathcal{O}(h^{2r+1}).\\
\end{aligned}
\end{equation*}
\end{proof} 
\begin{rem}
It follows from this theorem that our new methods  can be of
arbitrarily high order   only by   choosing a large integer $r$,
which is  very convenient  and simple.
\end{rem}

\section{Implementations}\label{sec-implementations}
This section considers the implementation issue  of the new methods.
\subsection{For the case that $G(y)$ is a constant matrix}
We first  consider the case that $G(y)$ is a constant matrix $M$.
Under this condition, the FFED method \eqref{EDCr} becomes
\begin{equation*}
\begin{aligned} &
\tilde{u}^{\prime}(\tau)=-M^{-1}\int_{0}^1 P_{\tau,\sigma}\nabla
U(\tilde{u}(\sigma))d\sigma,
\end{aligned}
\end{equation*}
which can be solved by the variation-of-constants formula  as
follows
\begin{equation*}
 \begin{aligned} &\tilde{u}(\tau)
= y_{0}-\tau h  M^{-1}  \int_{0}^1  \int_{0}^1 P_{\xi
\tau,\sigma}\nabla
U(\tilde{u}(\sigma))d\sigma d\xi\\
&= y_{0}-\tau h M^{-1} \sum_{i=0}^{r-1} \Big(
 \int_{0}^1  \tilde{\psi}_{i}(\xi \tau)d\xi\Big) \int_{0}^1 \tilde{\psi}_{i}(\sigma)\nabla
U(\tilde{u}(\sigma)) d\sigma.
\end{aligned}
\end{equation*}

 Following
\cite{Li_Wu(na2016)}, we introduce the generalized Lagrange
interpolation functions $l_{i}(\tau)\in X_{h}$ with respect to
$(r+1)$ distinct points $\{c_{i}\}_{i=1}^{r+1}\subseteq[0,1]$:
\begin{equation*}
(l_{1}(\tau),\ldots,l_{r+1}(\tau))=(\widetilde{\Phi}_{1}(\tau),\widetilde{\Phi}_{2}(\tau),\ldots,\widetilde{\Phi}_{r+1}(\tau))\Lambda^{-1},
\end{equation*}
where {$\{\Phi_{i}(t)\}_{i=1}^{r+1}$ is a basis of $X$,
$\widetilde{\Phi}_{i}(\tau)=\Phi_{i}(\tau h)$} and
\begin{equation*}
\Lambda=\left(\begin{array}{cccc}\widetilde{\Phi}_{1}(c_{1})&\widetilde{\Phi}_{2}(c_{1})&\ldots&\widetilde{\Phi}_{r+1}(c_{1})\\
\widetilde{\Phi}_{1}(c_{2})&\widetilde{\Phi}_{2}(c_{2})&\ldots&\widetilde{\Phi}_{r+1}(c_{2})\\ \vdots&\vdots& &\vdots\\
\widetilde{\Phi}_{1}(c_{r+1})&\widetilde{\Phi}_{2}(c_{r+1})&\ldots&\widetilde{\Phi}_{r+1}(c_{r+1})\\\end{array}\right).
\end{equation*}
   Then it follows from
\cite{Li_Wu(na2016)} that $\{l_{i}(\tau)\}_{i=1}^{r+1}$ is a basis
of $X_h$
 satisfying $l_{i}(c_{j})=\delta_{ij}.$ Since $\tilde{u}(\tau)\in
X_{h}$, $\tilde{u}(\tau)$ can be expressed by the basis of $X_{h}$
as
\begin{equation*}
\tilde{u}(\tau)=\sum_{i=1}^{r+1}\tilde{u}(c_{i})l_{i}(\tau).
\end{equation*}
Choosing $0= c_{1}<c_2<\cdots<c_{r+1}= 1$ and denoting
$y_{\sigma}=\tilde{u}(\sigma)$, we present the following practical
FFED  methods.

\begin{defi}\label{def PEDECr}
The practical  FFED  method  is defined  as follows:
\begin{equation}\label{PEDECr}
\left\{\begin{aligned} &y_{c_j}= y_{0}-c_jhM^{-1}  \sum_{i=0}^{r-1}
\Big(
 \int_{0}^1  \tilde{\psi}_{i}(\xi c_j)d\xi\Big) \int_{0}^1\tilde{\psi}_{i}(\sigma)\nabla
U(\sum_{m=1}^{r+1}y_{c_m}l_{m}(\sigma)) d\sigma,\quad j=2,\ldots,r,\\
&y_{1}= y_{0}- h M^{-1}  \sum_{i=0}^{r-1} \Big(
 \int_{0}^1  \tilde{\psi}_{i}(\xi  )d\xi\Big) \int_{0}^1\tilde{\psi}_{i}(\sigma)\nabla
U(\sum_{m=1}^{r+1}y_{c_m}l_{m}(\sigma)) d\sigma.
\end{aligned}\right.
\end{equation}

\end{defi}

In a similar way, for the stiff system with the special potential
$U$ of the form \eqref{spe-U}, we have the following practical EFFED
method for \eqref{EDCr-E}.
\begin{defi}\label{def PEDECr-E}
The practical  EFFED  method  is given  by
\begin{equation}\label{PEDECr-E}
\left\{\begin{aligned} &y_{c_j}=e^{-c_j h M^{-1}A}y_{0}-c_jhM^{-1}
\sum_{i=0}^{r-1} \int_{0}^1\Big(
 \int_{0}^1 e^{-(1-\xi)c_j h M^{-1} A} \tilde{\psi}_{i}(\xi c_j)d\xi\Big) \tilde{\psi}_{i}(\sigma)\nabla
V(\sum_{m=1}^{r+1}y_{c_m}l_{m}(\sigma)) d\sigma,\\
&\qquad \qquad\qquad\qquad\qquad\qquad\qquad\qquad \qquad \qquad\qquad\qquad\qquad\qquad\qquad\quad j=2,\ldots,r,\\
&y_{1}=e^{ -h M^{-1}A}y_{0}- hM^{-1} \sum_{i=0}^{r-1}
\int_{0}^1\Big(
 \int_{0}^1 e^{-(1-\xi)  h M^{-1} A} \tilde{\psi}_{i}(\xi  )d\xi\Big) \tilde{\psi}_{i}(\sigma)\nabla
V(\sum_{m=1}^{r+1}y_{c_m}l_{m}(\sigma)) d\sigma.
\end{aligned}\right.
\end{equation}
\end{defi}

It is noted that when $M=I$, this method has been proposed and
analysed  in \cite{17-new}.

\subsection{For the general case that  $G(y)$ depends on $y$}
Now we pay attention to the general case that $G(y)$ is a  matrix
depending on $y$. The FFED method \eqref{EDCr} becomes
\begin{equation*}
\begin{aligned} &
\tilde{u}^{\prime}(\tau)=-G^{-1}(\tilde{u}(\tau))\int_{0}^1
P_{\tau,\sigma}\nabla U(\tilde{u}(\sigma))d\sigma,
\end{aligned}
\end{equation*}
which implies
\begin{equation}
\begin{aligned}
&\tilde{u}(x)=y_0 -h\int_{0}^xG^{-1}(\tilde{u}(\tau))\int_{0}^1
P_{\tau,\sigma}\nabla
V(\tilde{u}(\sigma))d\sigma d\tau\\
=&y_0 -h\int_{0}^xG^{-1}(\tilde{u}(\tau))\int_{0}^1
\sum_{i=0}^{r-1}\tilde{\psi}_{i}(\tau)\tilde{\psi}_{i}(\sigma)\nabla
V(\tilde{u}(\sigma))d\sigma d\tau\\
=&y_0 -h\sum_{i=0}^{r-1}\int_{0}^xG^{-1}(\tilde{u}(\tau))
\tilde{\psi}_{i}(\tau)d\tau \int_{0}^1\tilde{\psi}_{i}(\sigma)\nabla
V(\tilde{u}(\sigma))d\sigma. \\
\end{aligned}
\end{equation}
Choosing $x=c_j$ for $j=2,\ldots,r+1$   and denoting
$y_{\sigma}=\tilde{u}(\sigma)$, we obtain the following practical
FFED method
\begin{equation}
\left\{\begin{aligned}
&\tilde{u}(\sigma)=\sum_{i=1}^{r+1}y_{c_{i}}l_{i}(\sigma),\\
 &y_{c_{j}}
= y_0 -h\sum_{i=0}^{r-1}\int_{0}^{c_{j}}G^{-1}(\tilde{u}(\tau))
\tilde{\psi}_{i}(\tau)d\tau \int_{0}^1\tilde{\psi}_{i}(\sigma)\nabla
V(\tilde{u}(\sigma))d\sigma, \ \ \  j=2,\ldots,r,\\
 &y_{1}
= y_0 -h\sum_{i=0}^{r-1}\int_{0}^{1}G^{-1}(\tilde{u}(\tau))
\tilde{\psi}_{i}(\tau)d\tau \int_{0}^1\tilde{\psi}_{i}(\sigma)\nabla
V(\tilde{u}(\sigma))d\sigma,
\end{aligned}\right.
\end{equation}
which represents a nonlinear system of equations for the unknowns $
y_{c_{j}}$ for $j=2,\ldots,r+1$ and it can be solved by iteration.

\begin{rem}
We note that the integrals appearing in the methods can be
calculated exactly for many cases.  If they cannot be directly
calculated, it is nature  to consider approximating them by a
quadrature rule.
\end{rem}

\section{Numerical test}\label{sec-nume}
As an example, we choose $$\varphi_k(t)=t^k,\ \ \ k=0,1,\cdots,r-1$$
for the function spaces $X$ and $Y$,  and then take $r=2$ and
$c_1=\frac{3-\sqrt{3}}{6},\ c_2=\frac{3+\sqrt{3}}{6}$ for our new
methods. We denote this method as FFED. In order to  show its
efficiency and robustness, we choose the following three methods in
the literature:
\begin{itemize}
\item  AVF:  the averaged vector field method studied
in \cite{McLachlan99};

\item  AVFC: the averaged vector field collocation method with $s=2$ given in
\cite{hairer2013};

\item  EEI:  the explicit exponential integrator of order four derived in
\cite{Hochbruck2009}.

\end{itemize}
 It is noted  that  we approximate
the integrals appearing in the methods by the four-point
Gauss-Legendre's quadrature. For implicit methods, we
  set $10^{-16}$ as the error tolerance and $10$ as the maximum number of each fixed-point iteration.

Consider the   gradient system  \eqref{IVPPP-G} with $$G(y)=\left(
    \begin{array}{cc}
      \cos(\theta) & \sin(\theta) \\
      -\sin(\theta) & \cos(\theta) \\
    \end{array}
  \right)^{-1}$$ and
$$U(y)=\frac{1}{2}r(y_1^2+y_2^2)-\frac{1}{2}\sin(\theta)\big(y_1y_2^2-\frac{1}{3}y_1^3\big)+\frac{1}{2}\cos(\theta)\big(-y_1^2y_2+\frac{1}{3}y_2^3\big),$$
where  $\theta=\pi/2-10^{-4}$ and $r=20.$ The  initial value is
chosen as $y_1(0)=0,\ y_2(0)=1.$ This system has been researched in
\cite{Mclachlan-98}.   We first solve it in $[0,100]$ with
$h=\frac{1}{10},\frac{1}{20},\frac{1}{100}$ and see Figure
\ref{p1-1}   for the results of the potential function $U(y)$.  Then
the system is solved   with $h=0.1/2^i$ for  $ i=1,\ldots,4$ and
$T=100,500,1000$. The global errors  are presented in Figure
\ref{p1-2}.

From these numerical results, it can be observed  that our FFED
method has a higher accuracy, a better energy-diminishing property,
and a more prominent damping behavior in comparison with the other
three methods.

\begin{figure}[ptb]
\centering
\includegraphics[width=4.5cm,height=7cm]{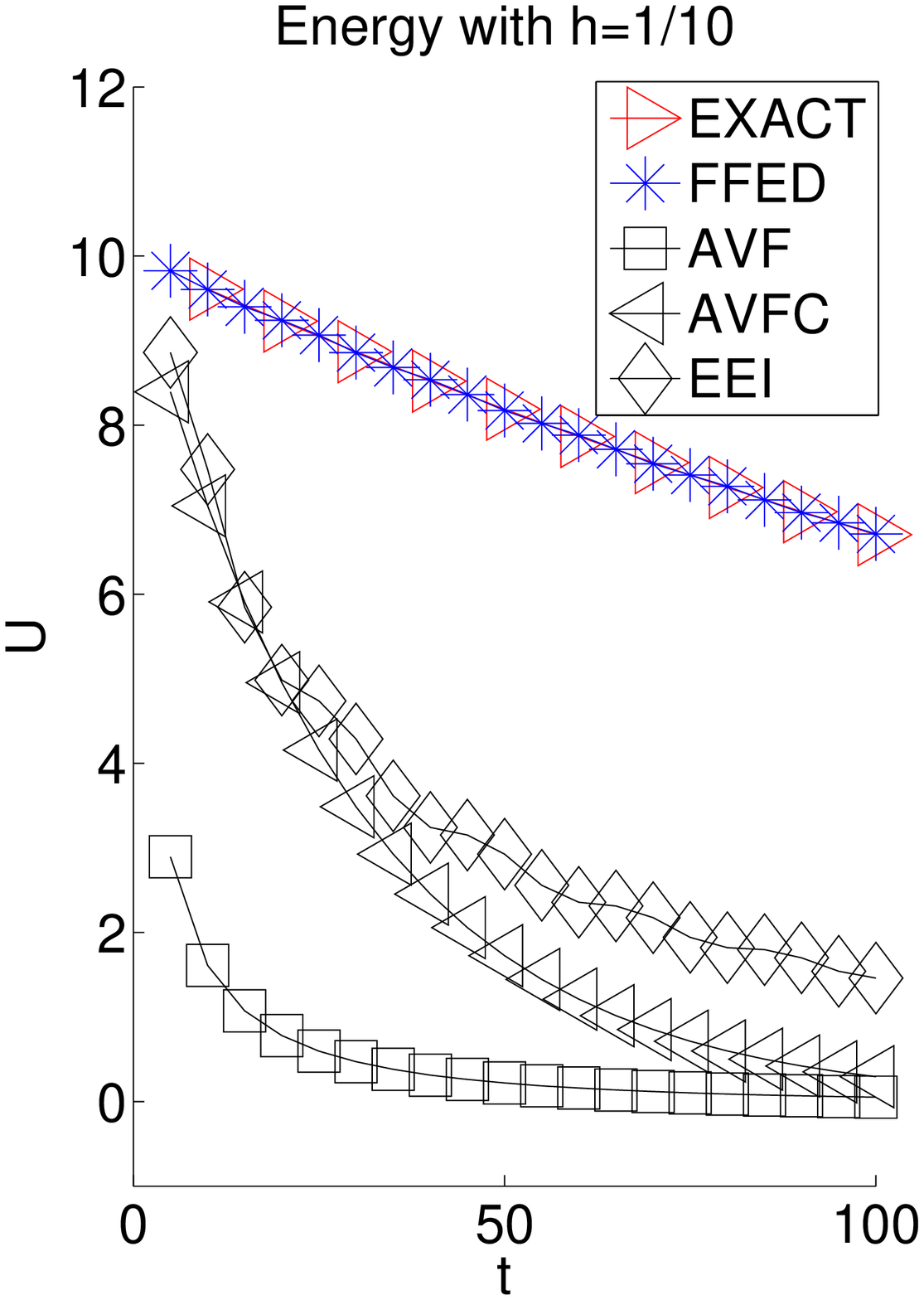}
\includegraphics[width=4.5cm,height=7cm]{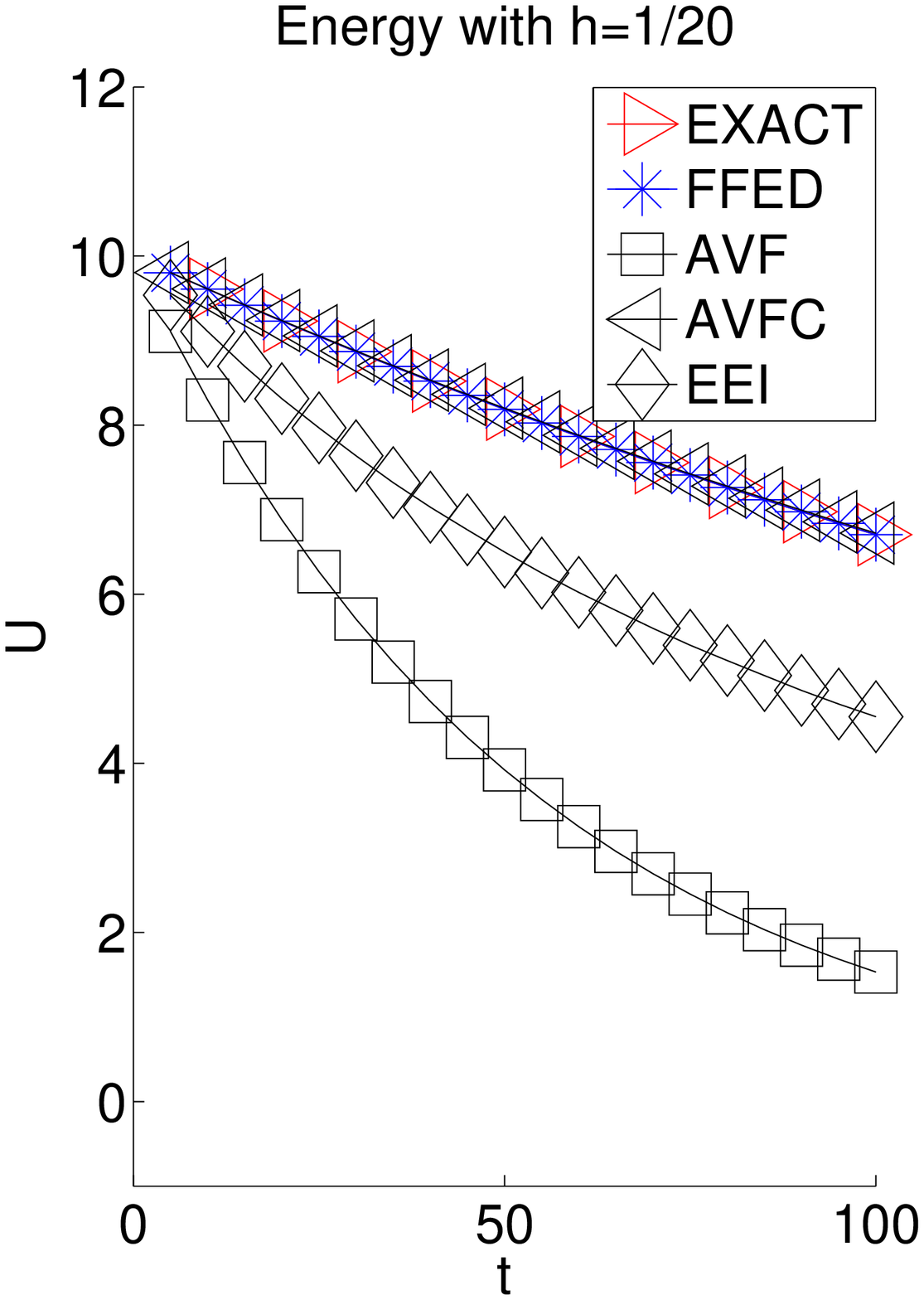}
\includegraphics[width=4.5cm,height=7cm]{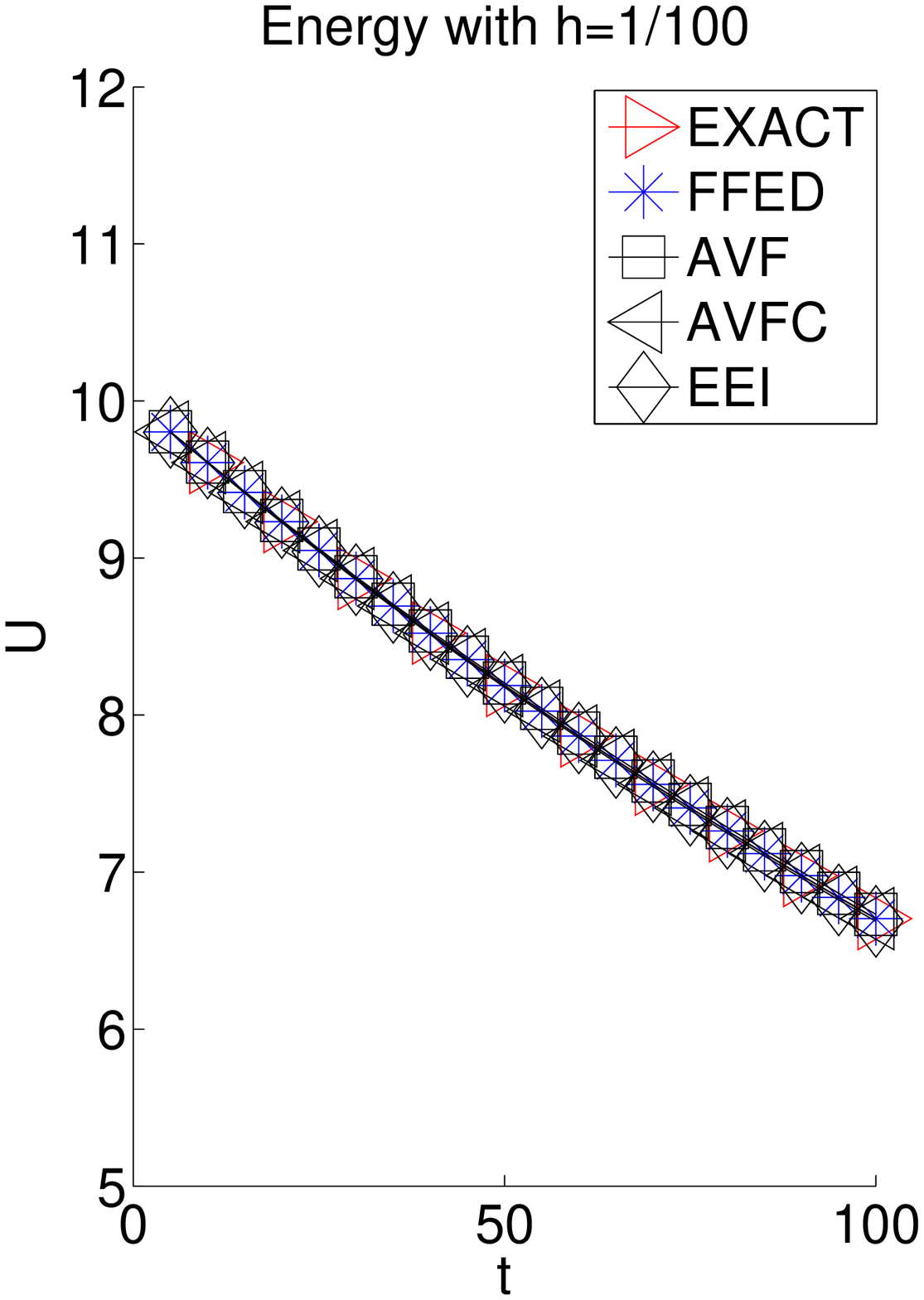}
\caption{The results of the energy $U$ against the time $t$.}
\label{p1-1}
\end{figure}

\begin{figure}[ptb]
\centering
\includegraphics[width=4.5cm,height=7cm]{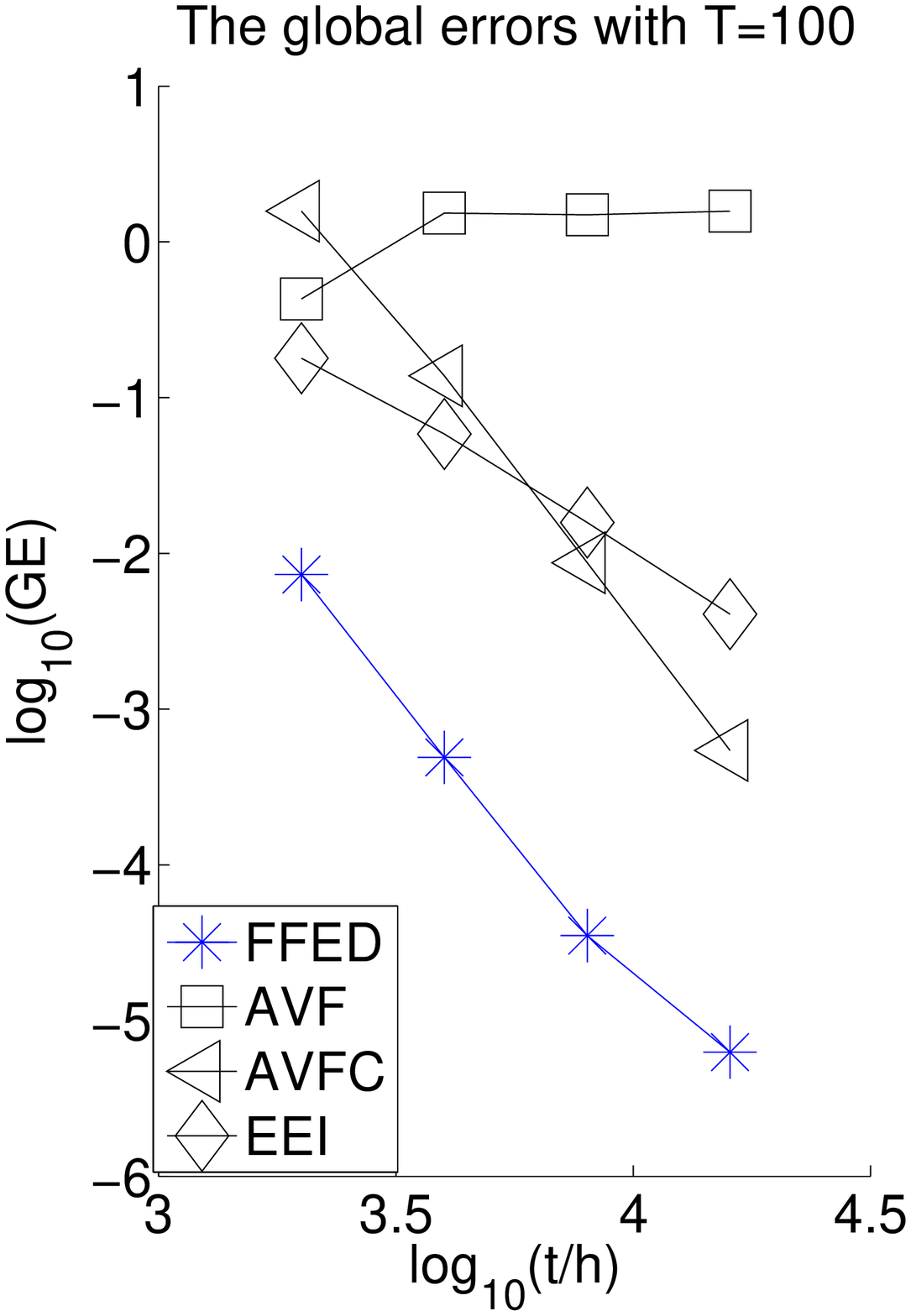}
\includegraphics[width=4.5cm,height=7cm]{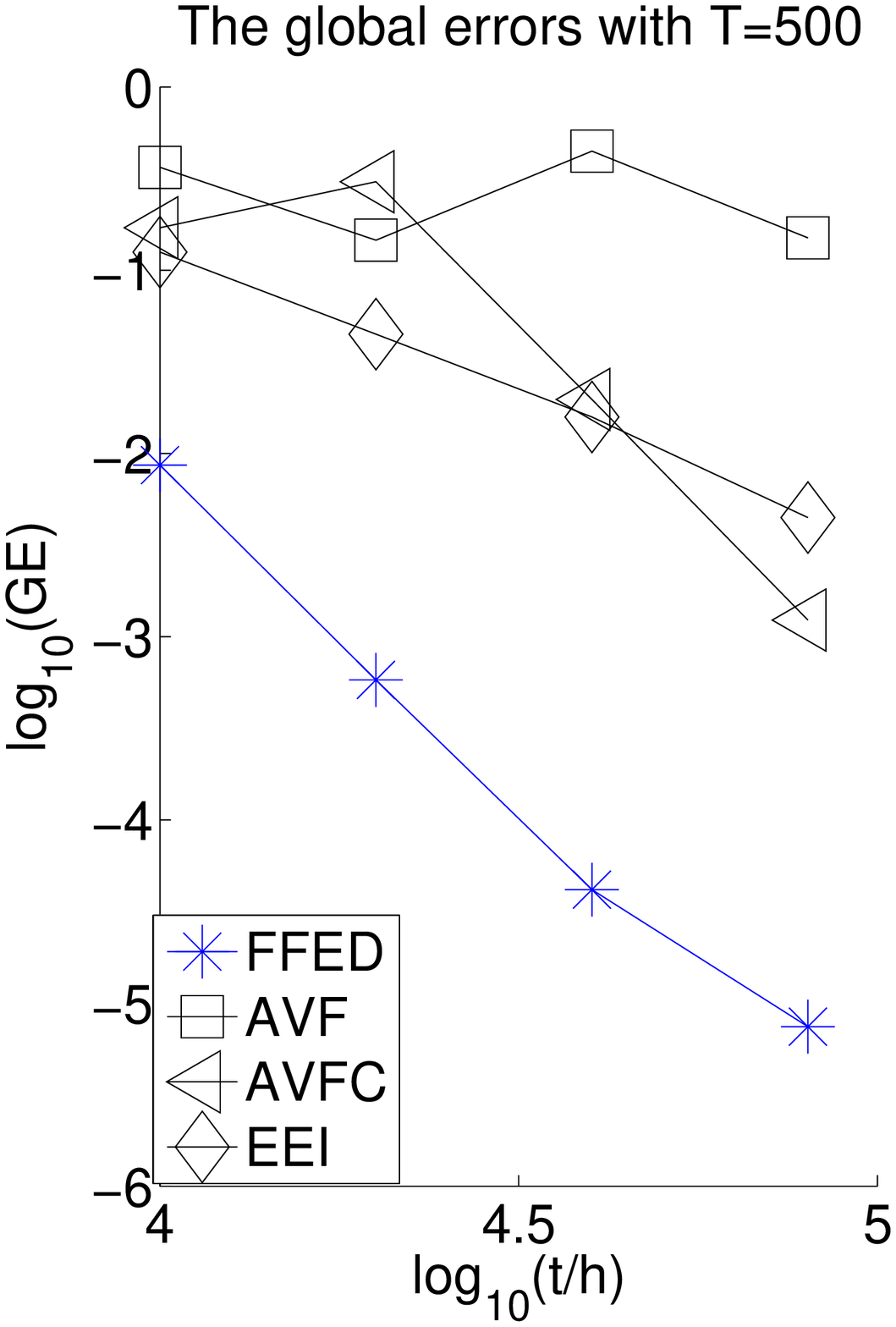}
\includegraphics[width=4.5cm,height=7cm]{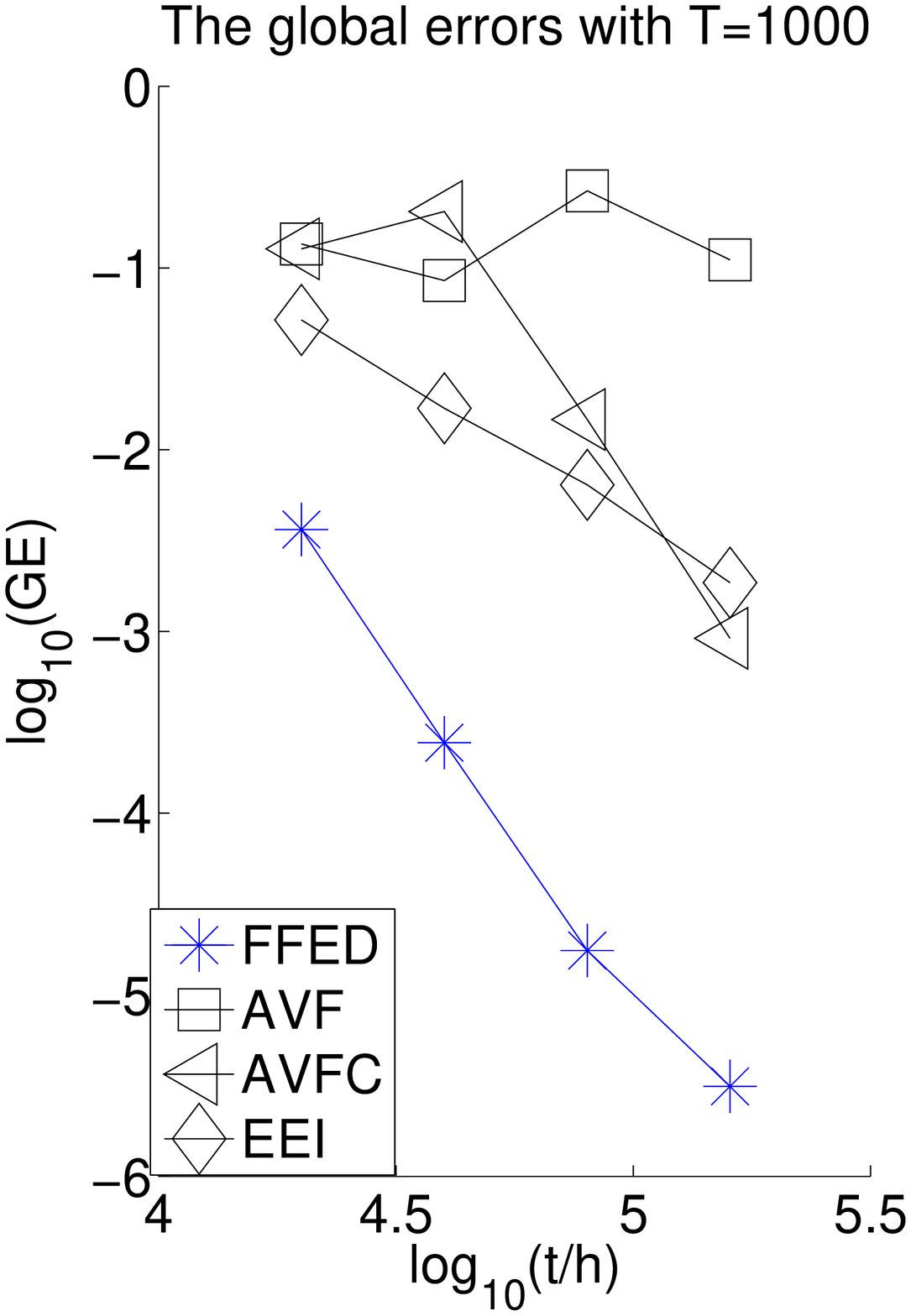}
\caption{The logarithm of the  global errors against the logarithm
of $t/h$.} \label{p1-2}
\end{figure}

\section{Conclusions}\label{sec-conc}
In this letter we derived a novel kind of functionally fitted
energy-diminishing methods for solving gradient systems. The
properties of the methods have been analysed.  It was shown that the
 arbitrary-order  methods are unconditionally energy-diminishing and  achieve
damping for stiff gradient systems. We also discussed the
implementations of the methods. The remarkable efficiency of the
 methods was demonstrated  by
a numerical test in comparison with three existing numerical methods
in the literature.

\section*{Acknowledgement}
The research of the first author is supported in part by the
Alexander von Humboldt Foundation and by the  Natural Science
Foundation of Shandong Province (Outstanding Younth Foundation).

\end{document}